\documentclass[12pt,twoside]{article} 
\usepackage{graphicx}
\usepackage{epsfig}
  
\usepackage{graphicx,amssymb}
\usepackage{amssymb,amsmath}
\usepackage{amsfonts}
\usepackage{amscd}
\usepackage{amsthm}
\newtheorem{Thm}{Theorem}[section]

\newtheorem{Definition}{Definition}[section]

\newtheorem{Ex}{Example}[section]
\date{} 

\begin{document}

\centerline{} 

\centerline{} 

\centerline{\Large{\bf$ \alpha $ $^m $ Continuous Maps in  Topological Spaces}}

\centerline{}

\centerline{\bf {Milby Mathew}} 

\centerline{} 

\centerline{Research Scholar} 

\centerline{Karpagam University} 

\centerline{Coimbatore-32, India} 

\centerline{milbymanish@yahoo.com}
\centerline{} 

\centerline{\bf {R. Parimelazhagan   }} 

\centerline{} 

\centerline{Department of Science and Humanities} 

\centerline{Karpagam College of Engineering} 

\centerline{Coimbatore-32, India} 

\centerline{pari\_kce@yahoo.com}
\centerline{} 

\centerline{\bf {S.Jafari   }} 

\centerline{} 

\centerline{College of Vestsjaelland South} 

\centerline{Herrestraedell} 

\centerline{4200 Slagelse,Denmark} 

\centerline{jafaripersia@gmail.com}

\begin{abstract}
The main aim of the present paper is to introduce new classes of functions called $ \alpha $ $^m $ continuous maps and $ \alpha $ $^m $  irresolute maps.  We   obtain some characterizations of these classes and properties are studied.
\end{abstract} 

{\bf Mathematics Subject Classification(2010):} 54 C10, 54C05 \\

{\bf Keywords:} $\alpha$$^m$-closed set,$ \alpha $ $^m $ -continuous maps,$ \alpha $ $^m $- irresolute maps

\section{Introduction} 
        Ganster and Reilly [9], Levine [10,11], Marcus [18], Mashour [16]et al have introduced LC-continuity, weak continuity, semi continuity ,quasi continuity,  $ \alpha $ -continuity.Balachandran [2] have introduced and studied generalized semi -continuous maps,semi locally continuous maps,semi-generalized locally continuous maps and generalised locally continuous maps. Maki  and  Noiri studied the  pasting lemma  for $ \alpha $ continuous mappings .  Milby and  R.Parimelazhagan  introduced and studied the properties of $ \alpha $ $^m $ -closed sets in topological spaces.   Crossley and Hildebrand[4]  introduced and investigated irresolute functions which are stronger than semi continuous maps but are indepedent of continuous maps.Since then several researchers have introduced several strong and weak forms of irresolute functions.Di Maio [6],Faro [7],Cammaroto and Noiri [3],  Maheswari  and  Prasad[13], Sundaram [23]  and  Palanimani[22]  have   introduced and studied quasi irresolute and strongly irresolute maps,strongly $ \alpha $ - irresolute maps,almost irresolute maps,$ \alpha $ -irresolute maps and gc- irresolute maps  are respectively.
        
\section{Preliminaries} 
 Before entering into our work, we recall the following definitions which are due to Levine.
 \begin{Definition}{[17]:}
 A subset A of a topological space $(X,\tau)$ is called a pre-open set if $ A\subseteq$ int(cl(A)) and pre-closed set if $cl(int(A))\subseteq A .$
 \end{Definition} 
 \begin{Definition}{[11]} 
  A subset A of a topological space $(X,\tau)$ is called a semi-open set if $A\subseteq cl(int(A))$ and semi closed set if $ int(cl(A))\subseteq A .$
 \end{Definition} 
  \begin{Definition}{[21]} 
 A subset A of a topological space $(X,\tau)$ is called an $\alpha$-open set if $A\subseteq int(cl(int(A)))$ and an $ \alpha$ -closed set if $cl(int(cl(A)))\subseteq A $.
 \end{Definition} 
 \begin{Definition}{[1]} 
 A subset A of a topological space $(X,\tau)$ is called a semi pre-open set $(\beta $-open set) if $A\subseteq cl(int(cl(A)))$ and semi-preclosed set $(\beta $ closed set) if  $int(cl(int(A)))\subseteq A $.
  \end{Definition} 
   \begin{Definition} {[12]}
  A subset A of a topological space $(X, \tau)$ is called a g-closed if $ cl(A) \subseteq U $ , whenever $A \subseteq U $ and U is open in $(X, \tau)$. 
  \end{Definition} 
    \begin{Definition}{[14]} 
  A subset A of a topological space $(X, \tau)$ is called a generalized $ \alpha$ -closed (briefly $g\alpha$ -closed) if $\alpha cl(A)\subseteq U$ whenever $ A \subseteq U$ and U is $\alpha$-open in $(X, \tau)$.
    \end{Definition}
  \begin{Definition}{[20]} A subset A of a topological space $(X, \tau)$ is called weakly generalized closed set (briefly wg closed ) if cl(int(A))$ \subseteq U $whenever $A \subseteq U $ and U is open in X.

    \end{Definition}
    \begin{Definition}{[15]}
 A subset A of a topological space $(X, \tau)$ is called weakly generalized  $\alpha$ -closed set (briefly $ wg\alpha$-closed ) if $  \tau ^\alpha- cl(int(A)) \subset U $ whenever $ A \subset U $ and U is $\alpha$ - open in $(X,\tau)$.
    \end{Definition}

\begin{Definition} A function  f:  ( X, $ \tau ) $ $\rightarrow $ ( Y, $ \sigma )$ from a topological space  X  into a topological space Y is called g -continuous if   f$^{-1}$ (V) is g closed in X for every closed set V of Y.
    \end{Definition}

\begin{Definition} A function  f:  ( X, $ \tau ) $ $\rightarrow $ ( Y, $ \sigma )$ from a topological space  X  into a topological space Y is called semi -generalised continuous  (briefly sg- continuous) if  f$^{-1}$ (V)  is sg- closed in X for evey closedset V of X.

    \end{Definition}
\begin{Definition} A functionf: $ X\rightarrow   Y$   is said to be  $ \alpha$- g continuous if  f$^{-1}$ (U) is    $\alpha$ g - open in X for each open set U of Y.

    \end{Definition}
\begin{Definition} A function f: $ X\rightarrow   Y$  is said to be weakly generalised continuous(wg-continuous) if the inverse image of every open set in Y is wg-open in X.

    \end{Definition}
\begin{Definition} A function f: $ X\rightarrow   Y$  is w-continuous if f$^{-1}$ (U) is w-openset in X for each openset U of Y. 

    \end{Definition}
\begin{Definition} A function  f:  ( X, $ \tau ) $ $\rightarrow $ ( Y, $ \sigma )$ from a topological space  X  into a topological space Y is called irresolute if  f$^{-1}$ (V) is semi closed in X for every semi-closed set V of Y.

    \end{Definition}
\begin{Definition}{[19]} A subset A of a topological space $(X,\tau)$ is called  $\alpha$$^m$-closed set if $ int(cl(A)) \subseteq U $   whenever $A \subseteq U $ and U is $\alpha$-open.

 \end{Definition}
\begin{center}
\section{ $\alpha$$^m$ continuous functions} 
\end{center}
 In this section we  introduce the concept of $\alpha$$^m$  continuous functions.
\begin{Definition} 
A map f: $ X\rightarrow   Y$ from a topological space X into a topological space Y is called $\alpha$$^m$ continuous if the inverse image of every closed set in Y is $\alpha$$^m$ closed set in X.
\end{Definition} 
\begin{Thm} 
 If a map f: $ X\rightarrow   Y$  from a topological space into a topological space Y is continuous then it is  $\alpha$$^m$ continuous but not conversely.
 \end{Thm} 
\begin{proof} 
Let  f: $ X\rightarrow   Y$  be continuous.Let F be any closed set in Y.Then the inverse image of     f$^{-1}$ (F) is closed in Y.Since every closed set is  $\alpha$$^m$ closed set(previous paper)[19],  f $^{-1}(F) $ is $\alpha$$^m$ closed in X.Therefore f is  $\alpha$$^m$ continuous.The converse need not be true as seen from the following example.
\end{proof}
\begin{Ex}
Let X=$ \{a,b,c\} $with topology     $ \tau $ =$ \{\Phi, X, \{a\}\}$ and $ Y=\{p,q \}$  and  $ \sigma$=$\{\Phi,\{p\},X\}$    Let f :(    X, $ \tau ) $   $ \rightarrow$   (Y,$ \sigma)$ be defined by f(a)=f(c)=q,f(y)=p.Then f is  $\alpha$$^m$ continuous .But f is not continuous.Since for the open set G=$\{p \}$  in Y,  f$^{-1}$(G)=$\{b\}$ is not open in X.        
  \end{Ex}
\begin{Thm} 
Let f:  ( X, $ \tau ) $ $\rightarrow $ ( Y, $ \sigma )$   be a map from a topological space(X,$\tau$) into a topological space  ( Y,$\sigma$)

(I)  The following statements are equivalent.

(a)  f is $\alpha$$^m$ continuous

(b)  The inverse image of each openset in Y is  $\alpha$$^m$ open in X.

(II) If f:( X, $ \tau ) $  $\rightarrow$  (Y,$\sigma)$  is  $\alpha$$^m$ continuous   then f(cl*(A))$ \subset$   $\overline { f(A)}$  for    every subset A of X;(here cl*(A)  as defined by Dunham[7].Further be defined  a topological $\tau$$^* $gclosure  by $\tau$$^* $ =$\{G:$ cl$^*$$($G$^c)$$=$G$^c$$\}$)

(III)The following statements are equivalent.

(a)For each point x $\in$ X   and each open set V containing f(x).Thereexist a  $\alpha$$^m$ open set U containing X suchthat f(U)$\subset$ V.

(b)For every subset A of X, f(cl*(A))$ \subset$   $\overline { f(A)}$ holds.

(c)The map  f:  ( X, $ \tau$$^* ) $ $\rightarrow $ ( Y, $ \sigma )$ from a topological space.( X, $ \tau$$^* ) $ defined by Dunham[7] into topological space  ( Y, $ \sigma )$ is continuous.
\end{Thm}

\begin{proof}

(I) Assume that  f: $ X \rightarrow Y $    is   $\alpha$$^m$ continuous. Let G be open in Y.Then G$^c$ is closed in Y.Since f  is $\alpha$$^m$ continuous ,    f$^{-1}$ (G$^c$) is $\alpha$$^m$  closed in X .   But     f$^{-1}$ (G$^c$)=X- f$^{-1}$ (G).Thus X- f$^{-1}$ (G) is  $\alpha$$^m$ closed in X and so f$^{-1}$ (G) is $\alpha$$^m$ open in X .
Therefore (a)  $\Rightarrow  $(b).

Conversely assume that the inverse image of each open set in Y is $\alpha$$^m$ open in X.Let F be any closedset in Y.Then$ F^c$ is open in Y.By assumption,f$^{-1}$ (F$^c$)  is  $\alpha$$^m$ open in X.But f$^{-1}$ (F$^c$) = X- f$^{-1}$ (F).Thus X- f$^{-1}$ (F) is  $\alpha$$^m$ open in X and so  f$^{-1}$ (F) is  $\alpha$$^m$ closed in X.
Therefore f is  $\alpha$$^m$ continuous .Hence  (b)  $\Rightarrow  $(a).Thus (a) and (b) are equivalent.

(II)Assume that  f  is   $\alpha$$^m$ continuous.Let A be any subset of X.Then $\overline { f(A)}$ is a closed set in Y.Since f is  $\alpha$$^m$ continuous  f $^{-1}$($\overline { f(A)})$ is  $\alpha$$^m$ closed in X and it contains A.But cl*(A) is the intersection of all  $\alpha$$^m$ closed sets containing A.Therefore cl*(A)$\subset$   f $^{-1}$($\overline { f(A)})$  and so   f(cl*(A))$ \subset$   $\overline { f(A)}$ .

(III)   (a) $\Rightarrow  $(b)
Let y$\in$ f(cl*(A)) and let V be any open neighbourhood of y.Then thereexist a point x $\in$ X and a $\alpha$$^m$ open set U suchthat f(x)=y,   x    $\in$    U,     x$\in$ cl*(A)    and     f(U)  $\subset$  V.      Since      x   $\in$  cl*(A),   U $\cap$ A   $\neq$ $\phi$    holds           and              hence     $f(A) \cap$ V   $\neq$ $\phi$.Therefore we have y=f(x)   $\in$($\overline { f(A)}$).

(b) $\Rightarrow  $ (a)
               Let   x  $\in$   X   and   V   be   any   openset    containing    f(x) .       Let  A=f$^{-1}$ (V$^c$),then x $\notin$A. Now cl*(A) $\subset$   f$^{-1}$ (f(cl*(A))) $\subset$ f$^{-1}$ (V$^c$)=A.   (ie)cl*(A) $\subset$ A. But A $\subset$  cl*(A) .Therefore A=cl*(A). Then ,since x $\notin $ cl*(A),thereexist a   $\alpha$$^m$ open set U containing x suchthat U $\cap$A =$\phi$ and hence f(U) $\subset$  f($A^c)$ $\subset $V.
 \end{proof}
\begin{Thm}
Let  f: $ X\rightarrow   Y$  be  $\alpha$$^m$  continuous map from a topological space X into a topological space Y and let H be a closed subset of Y.Then the restriction of f/H  :H  $ \rightarrow$ Y is  $\alpha$$^m$  continuous where H is endowed with the relative topology.
 \end{Thm}
\begin{proof}
Let F be  any  closed subset in Y.Since f is $\alpha$$^m$  continuous, f$^{-1}$ (F) is   $\alpha$$^m$  closed in X.Levine has proved that intersection of a closed set and also we  proved  that  intersection of  $\alpha$$^m$ closed sets  is   $\alpha$$^m$ closed. 
Thus if    f $^{-1}$  (F)  $\cap $   H = H  {$_1$}.then H {$_1$ } is a  $\alpha$$^m$  closed in X.
 since(f/H)$^{-1}$(F)= H$_1$ ,it is sufficient to proved that  H$_1$  is   $\alpha$$^m$  closed in H.Let G be any open set  of H suchthat   G$_1$ $\supset$  H$_1$ .
Let  G$_1$ =G $\cap$ H where G is open in X.
Now   H$_1$  $\subset$ G  $\cap$ H  $\subset$ G.Since  H$_1$  is   $\alpha$$^m$  closed in X,$ \overline  {H_1} $ $ \subset$ G.Now cl$_H$ (H$_1$)=$ \overline  {H_1} $  $\cap$ H $\subset$G $\cap$ H=G$_1$,where cl$_H$(A) is the closure of a subset A (CH) in a subspace H of X.Therefore f/H is  $\alpha$$^m$  continuous.

 \end{proof} 
\begin{Thm}

Let   $\alpha$=A $\cup$ B be a topological space with topology  $ \tau$ and Y be topological space with topology $\sigma$.
 Let f:(A, $ \tau$ /A) $\rightarrow$ (Y,$\sigma$)   and  g:(B, $ \tau$/B) $\rightarrow$ (Y,$\sigma) $    be  $\alpha$$^m$  continuous maps suchthat f(x)=g(x) for every x $\in$  A $\cap$  B.Suppose that A and B are  $\alpha$$^m$ closed sets in X.Then the combination     $\alpha$: (X, $ \tau$) $\rightarrow$ (Y,$\sigma$)  $\alpha$$^m$  continuous  .
 
  \end{Thm}

\begin{proof}
Let F be any closed set in Y. Clearly  $\alpha$$^{-1}$(F) =  f$^{-1}$ (F) $\cup$   g$^{-1}$ (F) = C $\cup$ D where C= f$^{-1}$ (F) and D = g$^{-1}$ (F).But C is  $\alpha$$^m$ closed in A and A is  $\alpha$$^m$ closed in X and so C is  $\alpha$$^m$ closed set in X;since we have proved that if B $\subset$ A $\subset$  X, B is  $\alpha$$^m$ closed in A and A is  $\alpha$$^m$ closed in X then B is  $\alpha$$^m$ closed in X.Similarly D is  $\alpha$$^m$ closed in  X. Also C $\cup$ D is   $\alpha$$^m$ closed in X.Therefore $\alpha$$^{-1}$(F)  is 
 $\alpha$$^m$ closed in X. Hence $\alpha$ is   $\alpha$$^m$  continuous .

 \end{proof} 
\begin{center}
\section{ $\alpha$$^m$-irresolute} 
\end{center}
\begin{Definition} 
A map f: $ X\rightarrow   Y$ from a topological space X into a topological space Y is called $\alpha$$^m$ - irresolute if the inverse image of every  $\alpha$$^m$ closed set in Y  is    $\alpha$$^m$ closed set in X.
\end{Definition} 
\begin{Thm}Let  f: $ X\rightarrow   Y$ ,g: $ Y\rightarrow   Z$ be two maps.Then  their composition gof:$ X\rightarrow   Z$  is $\alpha$$^m$ - continuous if f is $\alpha$$^m$ - irresolute and g is $\alpha$$^m$ - continuous.
  \end{Thm}

\begin{proof}
Let V be an open set in Z.Then (gof)$^{-1}$ (v)=( f$^{-1}$o g$^{-1}$)(v)= f$^{-1}$(v) where U= g$^{-1}$(v) is  $\alpha$$^m$ open in Y as g is  $\alpha$$^m$ -continuous.Since  f is  $\alpha$$^m$ -irresolute ,  f$^{-1}$(U) is  $\alpha$$^m$ open in X.Thus gof is  $\alpha$$^m$ continuous.
 \end{proof} 
\begin{Thm}Let  f: $ X\rightarrow   Y$,,g: $ Y\rightarrow   Z$ be two  $\alpha$$^m$ -irresolute  maps.Then their composition gof:$ X\rightarrow   Z$ is irresolutes maps.
 \end{Thm}
\begin{proof}Let V be an $\alpha$$^m$- open set in Z.Consider (gof)$^{-1}$ (v)= f$^{-1}$(v) where U= g$^{-1}$(v) is  $\alpha$$^m$ open in Y as g is  $\alpha$$^m$ -irresolute.Since  f is  $\alpha$$^m$ -irresolute, f$^{-1}$(U) is $\alpha$$^m$ open in X.Thus  gof is   $\alpha$$^m$ open.

 \end{proof}
\begin{Thm}Let  f: $ X\rightarrow   Y$ ,g: $ Y\rightarrow   Z$ be two maps such that gof:$ X\rightarrow   Z$ is  $\alpha$$^m$ closed map.Then(i) if f is continuous and surjective then g is $\alpha$$^m$ closed and (ii) if g is irresolute and injective then f is $\alpha$$^m$ closed .
 \end{Thm}
\begin{proof}Let H be closed set in Y.Since f$^{-1}$(H) is closed in X. (gof)( f$^{-1}$(H)) is  $\alpha$$^m$ closed set in Z and hence g(H) is $\alpha$$^m$ closed in Z.Thus g is $\alpha$$^m$ closed.Hence(i).
Let F be closed set of X.Then (gof)(F) is $\alpha$$^m$ closed in Z and g$^{-1}$(gof)(F) is  $\alpha$$^m$ closed  in Y.Since g is injective f(F)= g$^{-1}$(gof)(F) is $\alpha$$^m$ closed  in Y.Then f is $\alpha$$^m$ closed .
\end{proof}

\end{document}